\documentclass[11pt]{amsart}
\usepackage{graphicx}
\usepackage{amssymb,amscd,amsthm,amsxtra}
\usepackage{latexsym}
\usepackage{epsfig}
\newtheorem{thm}{Theorem}[section]
\newtheorem{cor}[thm]{Corollary}

\newtheorem{prop}[thm]{Proposition}
\theoremstyle{definition}

\theoremstyle{remark}

\numberwithin{equation}{section}
\newcommand{\R}{\mathbb R}

\newcommand{\eps}{\epsilon}
\newcommand{\Om}{\Omega}
\newcommand{\ov}{\overline}
\newcommand{\p}{\partial}
\newcommand{\f}{\varphi}

\newcommand{\comment}[1]{}
\begin{document}

\title[Stable cones]{ Some remarks on stability of cones for the one-phase free boundary problem}

\author{David Jerison}
\address{Department of Mathematics, 77 Massachusetts Ave, Cambridge, MA  02139-4307}\email{\tt jerison@math.mit.edu}
\author{Ovidiu Savin}
\address{Department of Mathematics, Columbia University, New York, NY 10027}\email{\tt  savin@math.columbia.edu}

\thanks{D. J. was supported by the Bergman Trust
and an NSF grant DMS-1069225.}
\thanks{ O.~S.~ was supported by NSF grant DMS-1200701.}

% ----------------------------------------------------------------
\begin{abstract}
We show that stable cones for the one-phase free boundary problem are hyperplanes
in dimension $4$. As a corollary, both one and two-phase energy minimizing 
hypersurfaces are smooth in dimension $4$.
\end{abstract}
\maketitle
% ----------------------------------------------------------------

\section{Introduction}

We investigate stable homogeneous solutions $$u: \ov \Omega \to \R, \quad \quad \Om \subset \R^n,$$ to the one-phase free boundary problem
\begin{equation}\label{op}
\triangle u=0 \quad \mbox{in $\Om$}, \quad \quad u=0 
\ \mbox{ and } \ |\nabla u|=1 \quad \mbox{on $\p \Om \setminus \{0\}$}.
\end{equation}
Here $u$ is a homogeneous of degree one function which is positive in $\Omega$, and $\Omega$ is a conical domain ($\Om = r\Om$ for all $r>0$) with smooth cross-section.

We are interested in solutions $u$ that 
are stable with respect to the Alt-Caffarelli (see \cite{AC}) energy functional,
\begin{equation}\label{Eu}
E(u,B)=\int_B |\nabla u|^2 + \chi_{\{u>0\}} \, dx ,
\end{equation} 
with respect to compact domain deformations that do not contain the origin.  
Explicitly, the 
stability we require is that for any smooth vector field $\Psi:\R^n
\to \R^n$ with $0 \notin \mbox{supp} \, \, \Psi \subset B_R$ 
we have
\begin{equation} \label{stab}
\frac{d^2}{d t^2} \, \, 
E \left (u(x+ t \Psi(x)), B_R \right ) \ge 0 \quad \quad \mbox{at $t=0$}.
\end{equation}

There is a vast literature concerning the one-phase free boundary problem;
see, for example, the book by Caffarelli and Salsa \cite{CS}. Many results in 
the regularity theory of the free boundary $\p \{u>0\}$ parallel the 
corresponding statements in the regularity theory of minimal surfaces, 
see \cite{C1,C2,DJ2,W}.

Our main result is the following.

\begin{thm}\label{d4}
The only stable homogeneous solutions in dimension $n \le 4$ are the one-dimensional solutions $u=(x \cdot \nu)^+$ for unit vectors $\nu$. 
\end{thm}

For dimension $n = 3$ this result was obtained by Caffarelli, Jerison and Kenig in \cite{CJK}, and they conjectured that it remains true up to dimension $n\le 6$. On the other hand De Silva and Jerison provided in \cite{DJ1} an example of a nontrivial minimal solution in dimension $n=7$.  

The main consequence of Theorem \ref{d4} is that it implies the smoothness of the free boundary for minimizers in both the one-phase and {\it two-phase problem} in dimension $n \le 4$. Moreover, by the dimension reduction arguments of Weiss \cite{W}, we obtain the following regularity result.

\begin{cor}
 Let $v$ be a minimizer of the energy functional $$J(v):=\int_{B_1} \left(|\nabla v|^2 + Q_+(x) \chi_{\{v>0\}} + Q_-(x) \chi_{\{v \le 0\}} \right) dx$$ with $Q_\pm$ smooth functions satisfying $Q_+ > Q_-$. Then the free boundary $$F(v):=\p \{v>0\} \cap B_1$$ is a smooth hypersurface except possibly on a closed singular set $\Sigma \subset F(v)$ of Hausdorff dimension $n-5$, and 
 $$(v_\nu^+)^2-(v_\nu^-)^2=Q_+-Q_- \quad \quad \mbox{on $F(v) \setminus \Sigma$}.$$  
 \end{cor}

Our proof of Theorem \ref{d4} is similar to James Simons's proof (see \cite{S})
of rigidity 
of stable minimal cones in low dimensions: we find a function
involving the second derivatives of $u$ which satisfies a differential
inequality for the linearized equation.  In particular, the
proof in dimension 3 is not the same as the one in \cite{CJK}.

The paper is organized as follows. In Section \ref{s2} we collect
some basic facts about stability and the linearized equation of
$u$. In Section \ref{s3} we obtain the differential inequality for a
function involving $\|D^2u\|$ and deduce the rigidity result in
dimension $n=3$. Finally in Section \ref{s4} we treat the case
$n=4$ by modifying the function considered in Section \ref{s3}.

\smallskip

\section{Preliminaries and stability}\label{s2}

In this section we recall some facts about stability of solutions $u$ of \eqref{op} that were obtained in \cite{CJK}. We insist more on the non-variational approach to stability.   

\subsection{Normals for second derivatives at the boundary}

Fix a point $$x_0 \in \p \Om \setminus \{0\}$$ and choose a system of coordinates at $x_0$ such that $$ e_n=\nu_{x_0} \quad \mbox{the interior normal at $x_0$}$$ and $\p \Om$ is given locally by the graph of a function $g$  
$$\Om =\{x_n >g(x')  \}, \quad \mbox{with} \quad \nabla_{x'}g(x'_0)=0, \quad D_{x'}^2g(x_0') \quad \mbox{diagonal.}$$
By differentiating $u(x',g(x'))= 0$ in the $i$, $j$ directions, $i$, $j<n$, we obtain
\begin{equation}\label{1}
u_i=0, \quad \quad u_{ij}=-u_n \, g_{ij}=-g_{ij} \quad \quad \mbox{at $x_0$},
\end{equation}
where subscripts indicate partial derivatives.   Differentiating
$|\nabla u|^2(x',g(x')) = 1$, we obtain
\[
\sum_{k=1}^n u_k \, u_{ki}=0, \quad \quad \sum_{k=1}^n (u_k \, u_{kij}+
u_{ki}\, u_{kj})=-u_n\,  u_{nn} \, g_{ij}
\]
for $i$, $j<n$.  In conclusion, applying $u_n=1$ and \eqref{1} at $x_0$ we have, 
\[
u_{in}=0 \quad \mbox{at} \ x_0, \quad i<n.
\]
Consequently, $D^2u$ is diagonal at $x_0$, and
\begin{align}\label{2}
\nonumber u_{ijn}&= \, 0 \quad \quad \quad \quad \quad \quad  
\qquad \mbox{if $i$, $j<n$ and $i \ne j$,}\\
u_{iin}&=\, u_{nn}\, u_{ii} -u^2_{ii} \quad \quad \quad  \  \mbox{for each $i<n$,}\\
\nonumber u_{nnn}&= \sum_{k =1}^n  u_{kk}^2,
\end{align}
where the last equation follows from the sum over $i$ of the previous one 
and $\triangle u = \triangle u_n=0$.

\subsection{The linearized equation}
A smooth function $v:\ov \Om \to \R$ solves the linearized equation for a solution $u$ if 

\begin{equation}\label{le}
\left \{  \begin{array}{l}
\triangle v =0 \quad \quad \quad \mbox{in $\Om$,}\\
\  \\
v_\nu=u_{\nu\nu} \, v \quad \quad \mbox {on $\p \Om \setminus \{0\}$.}
\end{array}
\right.
\end{equation}
Notice that from $\triangle u=0$ and \eqref{1} it follows that
$$-u_{\nu \nu}=H$$ where $H$ denotes the mean curvature of $\p \Om$ oriented towards the complement of $\Om$. Thus the second equation in \eqref{le} can be rewritten as
$$v_\nu + H\, v=0  \quad \quad \mbox{on $\p \Om$}.$$

In the case when $\Omega$ is a cone different from a half-space, it easily follows that $$H>0.$$ Indeed, $|\nabla u|^2/2$ is a subharmonic function homogeneous of degree $0$, and its maximum occurs on the boundary. Then either $|\nabla u|^2/2$ is constant or by Hopf lemma its normal derivative on $\p \Om$, which equals $-H$, is negative.  

The linearized equation \eqref{le} is obtained by requiring that $(u+\eps v)^+$ solves the original equation up to an error of order $O(\eps^2)$ (here we think that $u$ and $v$ are extended smoothly in a neighborhood of $\p \Om$). Thus the function $v$ above represents the infinitesimal vertical distance between the graph of a perturbed solution and the graph of the original solution $u$ of \eqref{op}.

We deduce briefly \eqref{le}. The interior condition for $v$ is obvious. For the boundary condition we see that the free boundary of $(u+\eps v)^+$ lies in $O( \eps^2)$ of the surface $\Gamma_\eps$ obtained as
$$x \in \Gamma_0:=\p \Omega \quad \longmapsto \quad x_\eps \in \Gamma_\eps, \quad \quad x_\eps:=x-\eps  \, v(x) \, \nu_x.$$
Thus
\begin{align*}
\nabla (u+\eps v)(x_\eps)= &\nabla u(x)- D^2 u(x) (x_\eps-x) + \eps\, \nabla v(x_\eps) + O(\eps^2) \\ 
=&\nu - \eps \, v \, (D^2 u) \,  \nu + \eps \, \nabla v(x) + O(\eps^2)
\end{align*}
and
$$| \nabla (u+\eps v)(x_\eps)|^2=1+2 \eps (v_\nu-v \, u_{\nu \nu}) +O(\eps^2),$$
which gives the second condition in \eqref{le}.

Clearly the directional derivatives $v = e \cdot \nabla u$ 
solve the linearized equation, since they arise from translation of the
solution $u$.  The boundary equation can also be seen directly 
in the coordinates of Section 2 for which $D^2u$ is diagonal at $x_0$:   
$v = e\cdot \nu$ and $v_\nu = v_n = (e\cdot\nu) u_{nn} = u_{\nu\nu} v$.

\subsection{Criteria for stability and instability}  Let $u$ be 
a homogeneous one-phase free boundary solution $u$ as in \eqref{op} 
supported on the cone $\Om$.  Consider the annulus
\[
\mathcal U =  \{x\in \R^n:  0< c_1 < |x| < c_2\}
\]
The main lemma of \cite{CJK} says
that the stability \eqref{stab} under perturbations in $\mathcal U$ 
implies that for all smooth functions $f$ supported in $\mathcal U$,
\begin{equation}\label{stab1} 
\int_{\p \Om} Hf^2 d\sigma \le \int_{\Om} |\nabla f|^2 dx \, .
\end{equation}
We will deduce from \eqref{stab1} a criterion for instability in the form 
we will need, that is, expressed in terms of subsolutions.

We say that $v$ is a subsolution to the linearized equation
\eqref{le} in $\Om \cap \mathcal U$ if 
 \begin{equation}\label{si}
\left \{  \begin{array}{l}
\triangle  v \ge 0 \quad \quad \quad \mbox{in $\Om \cap \mathcal U$,}\\
\  \\
 v_\nu + H \, v \ge 0 \quad \quad \mbox {on $\mathcal U\cap \p \Om$,}
\end{array}
\right.
\end{equation}
with
\[
v \ge 0 \quad \mbox{on $\Om \cap \mathcal U$,} \quad \quad   v=0 \quad \quad \mbox{on $\Om \cap \p \mathcal U$.}  
\]

It follows from integration by parts that if there is a 
{\em strict} subsolution $v$ as in \eqref{si}, then $u$ is 
unstable.  Indeed, 
\begin{equation}\label{byparts}
\int_\Om |\nabla v|^2 dx 
 = -\int_\Om v \triangle v \, dx - \int_{\p \Om} v v_\nu \, d\sigma 
 \le - \int_{\p \Om} v v_\nu \, d\sigma  \le \int_{\p \Om} Hv^2  \, d\sigma  
\end{equation}
If either inequality in \eqref{byparts} is strict, then $u$ is unstable in $\mathcal U$. 

We prove Theorem \ref{d4} by constructing an explicit subsolution $v$ 
to \eqref{si} which depends on the second derivatives of $u$.  The
function $v$ is a product of spherical and radial parts.
Denote by $\Om_S$ the intersection of $\Om$ with
the unit sphere and write $\triangle_S$ for the Laplacian
on the sphere.  The following result is implicit in \cite{CJK},
but not stated or used directly there.
\begin{prop} \label{p0.5}   Suppose there is a nonnegative function 
$\f$ defined on $\bar \Om_S$ that is a strict subsolution
to 
  \begin{equation*} 
\left \{  \begin{array}{l}
\triangle_S \f \ge \lambda \f  \qquad \mbox{in $\Om_S$,}\\
\  \\
\f_\nu + H \, \f  \ge 0 \quad \quad \mbox {on $\p \Om_S$,}
\end{array}
\right.
\end{equation*}
and suppose that the constant $\lambda$ satisfies
\[
\lambda \ge (n-2)^2/4.
\]
Then $u$ is unstable in the sense that \eqref{stab} fails for
some perturbation $\Psi$ in a sufficiently large annulus.
\end{prop}
\begin{proof}  Define $\Lambda$ by 
\begin{equation}\label{inf}
-\Lambda:= \inf_{\psi }  \frac{\int_{\Om_S}|\nabla \psi|^2 - \int_{\p \Om_s}H \psi^2  }{\int_{\Om_S} \psi^2 } .
\end{equation}
As in \eqref{byparts}, an integration by parts and the assumption
that $\f$ is a strict subsolution yields
\[
\int_{\Om_S} |\nabla_S\f|^2 - \int_{\p \Om_S} H \f^2 < - \lambda \int_{\Om_S} \f^2,
\]
so that $\Lambda > \lambda$.   It is  well known that the minimizer $\bar \psi$
of \eqref{inf}
exists and satisfies $\bar \psi \in C^\infty(\bar \Om_S)$, $\bar \psi > 0$ in $\Om_S$,
and
\[
\triangle_S \bar \psi = \Lambda \bar \psi \ \mbox{ on $\Om_S$}; \quad
\bar \psi_\nu  + H \bar \psi = 0 \ \mbox{ on $\p \Om_S$}.
\]
Extend $\bar \psi$ to be homogeneous of degree $0$ on $\Om$
and define $v: = f(r) \bar \psi$; $r = |x|$.   Then 
\[
\triangle v= (f'' +(n-1)f'/r  + \Lambda f/r^2) \bar \psi.
\]
On the other hand, it is straightforward to check that if $f$ satisfies the constant coefficients ODE 
$$f''+\alpha f'/r +  \beta f/r^2 =0,$$
then $f$ oscillates around $0$ if and only if $$4 \beta > (\alpha -1)^2.$$
Let $\alpha= n-1$. Since $\Lambda > \lambda \ge (n-2)^2/4$, we may 
choose $\beta$ so that
\[
\Lambda > \beta >  (\alpha-1)^2/4 = (n-2)^2/4,
\]
and let $\mathcal U$ be the annular region between 
two consecutive zeros of $f$ where $f$ is positive. Then 
$v>0$ on $\Om \cap \mathcal U$, and
\[
\triangle v = (\Lambda - \beta) f\bar \psi/r^2 >0 \ \mbox{in $\Om \cap \mathcal U$.}
\]
Moreover, since $f$ is radial, $v_\nu + H v = 0$ on $\p \Om$
and $v=0$ on $\Om\cap \p \mathcal U$ because $f=0$ on $\p \mathcal U$.
Therefore $v$ is a strict subsolution for \eqref{si}, and $u$ is unstable.
\end{proof}

It remains to find the function $\f$.  It will turn out
that $\f$ is constructed using functions that are homogeneous of
degree $-\mu \neq 0$, so we will rewrite 
Proposition \ref{p0.5} as follows.
 \begin{prop}\label{p1} 
If there exists $\bar v \ge 0$, homogeneous of degree $-\mu$ on $\Om$, 
that is a strict subsolution for the following problem 
  \begin{equation}\label{si2}
\left \{  \begin{array}{l}
\triangle \bar v \ge \gamma \,
\bar v/ |x|^2 \quad \quad \quad \mbox{in $\Om$,}\\
\  \\
\bar v_\nu + H \,\bar v \ge 0 \quad \quad \mbox {on $\p \Om \setminus \{0\}$,}
\end{array}
\right.
\end{equation}
and the constant $\gamma$ satisfies
\begin{equation}\label{gi}
\gamma \ge \left(\frac n2 -1 -\mu \right)^2,
\end{equation}
then $u$ is unstable 
%(By strict subsolution, we mean that equality cannot hold everywhere
%in \eqref{si2}.)
\end{prop}
Note that \eqref{si2} is equivalent to 
 \begin{equation}\label{si3}
\left \{  \begin{array}{l}
\triangle (\log \bar v) + |\nabla (\log \bar v)|^2 \ge
\gamma /|x|^2 \quad \quad \quad \mbox{in $\Om \cap \{\bar v>0\}$,}\\
\  \\
\frac 1 H (\log \bar v)_\nu   \ge -1 \quad \quad \mbox {on $\p \Om \cap \{\bar v>0\}$.}
\end{array}
\right.
\end{equation}
\begin{proof}  The function $\bar v$ satisfies
\[
\triangle_S \bar v \ge (\gamma + \mu(n-2-\mu))\bar v \ \mbox{ on $\Om_S$}
\]
and the condition
\[
\gamma + \mu(n-2-\mu) \ge (n-2)^2/4
\]
is the same as \eqref{gi}.
\end{proof}

Although we do not need this is the sequel, we remark 
that  the sufficient conditions stated here for instability are
also necessary, as shown in the following proposition.
\begin{prop} \label{equiv}  The following are equivalent.

a) The stability inequality \eqref{stab1} 
holds for all $f\in C_0^\infty(\mathcal U)$.

b) There exists $\bar f$ satisfying $\bar f>0$ and 
$\triangle \bar f = 0$
in $\Om \cap \mathcal U$ and the boundary condition
\[
\bar f_\nu + H \, \bar f =  0 \quad \mbox {on \  $\mathcal U\cap \p \Om$}.
%\qquad 
%\bar f = 0 \quad \mbox{on $\Om \cap \p \mathcal U$}
\]

c) There are no nonnegative strict subsolutions as in \eqref{si},
on any annulus $\mathcal U' \subset \subset \mathcal U$, 
that is, no $v\ge0$ strict subsolutions to 
\[
\triangle v\ge 0 \quad \mbox{in \ $\Om \cap \mathcal U'$};
\qquad v_\nu + H \, v \ge 0 \quad \mbox {on \ $\mathcal U'\cap \p \Om$};
\qquad 
v = 0 \quad \mbox{on \ $\Om \cap \p \mathcal U'$.}
\]
\end{prop}
\begin{proof}  To prove that a) implies b), note that the minimizer
$\bar f$ to 
\[
\inf_{f \in C_0^\infty(\mathcal U)} 
\frac{\int_\Om |\nabla f|^2 - \int_{\p \Om} Hf^2}
{\int_{\Om} f^2}
\]
satisfies the required properties.   To prove that b) implies c), observe
that if $v$ existed, then $\triangle (v/\bar f) \ge 0$ on 
$\Om\cap \mathcal U'$ and $(v/\bar f)_\nu\ge0$ on $\p (\Om\cap \mathcal U')$ 
so that $v/\bar f$ is constant.  But $v$ cannot be a multiple of $\bar f$. 

Finally, we prove that c) implies a) by establishing the 
contrapositive.   Suppose that a) is false. 
Then there exists a slightly smaller annulus
$\mathcal U' \subset \subset \mathcal U$ for which 
\[
-\delta: = \inf_{f\in C_0^\infty(\mathcal U')} 
\frac{\int_\Om |\nabla f|^2 - \int_{\p \Om} Hf^2}
{\int_{\Om} f^2} < 0
\]
The minimizer $g$ is a nonnegative strict subsolution in 
$\Omega \cap \mathcal U'$. The strictness follows from 
$\triangle g = \delta g > 0$.  This shows that c) does 
not hold. 
\end{proof}

Proposition \ref{equiv} says in particular that the stability of a solution 
$u$ in a region is equivalent to the existence of a positive 
solution to the linearized equation in the same region.
In fact, in non-variational elliptic problems this characterization can be 
taken as the definition of stability.  Typically in such
non-variational problems, when such a positive solution 
exists, then, in a neighborhood of the graph of $u$,
the space can be foliated by perturbed solutions.  By contrast,
the existence of a strict subsolution is
essentially equivalent to saying that solutions 
to the linearized equation must change sign
and corresponds in the nonlinear setting to the case 
when the graph of $u$ and the graph of ``nearby'' perturbed solutions
``cross each other."

 \section{The case $w=\|D^2u\|$}\label{s3}
 
 In this section we show that $$\bar v= w^\alpha$$ satisfies an inequality of the type \eqref{si2} where $w:=\|D^2u\|$, that is
 $$w^2:=\|D^2u\|^2= \sum_{i,j=1}^nu^2_{ij}.$$

 \subsection{The interior inequalty}
 First we obtain an inequality for harmonic functions which is similar to 
Simons's inequality for minimal surfaces.
 
 \begin{prop}\label{SI}
 Assume $u$ is harmonic and homogeneous of degree 1. Then
 $$w \, \triangle w \ge \frac {2}{n-1} \, |\nabla w|^2 + 2 \,  \frac{n-2}{n-1} \, \frac{w^2}{|x|^2},$$
 in the set $\{w>0\}$.
 \end{prop}
 
 \begin{proof}
 We have 
 $$w \, w_k=\sum_{i,j=1}^n u_{ij}\, u_{ijk}  \quad \mbox{for each $k=1,..,n$},$$
 and 
 \begin{equation}\label{31}
 w \, \triangle w + |\nabla w|^2= \sum_{i,j,k=1}^n 
 (u^2_{ijk} + u_{ij}\,u_{ijkk})
 = \sum_{i,j,k=1}^n  u_{ijk}^2.
 \end{equation}
 Since $u$ is homogeneous of degree one, the radial direction $x/|x|$ is an eigenvector for $D^2u$. We choose a system of coordinates such that $e_1$ points in the radial direction at $x$. Then
 $$u_{1i}=0 \quad \mbox{for each $i=1,..,n$} ,$$
 and since $u_{ij}$ are homogeneous of degree $-1$ we obtain
 $$u_{1ij}=-\frac{u_{ij}}{|x|}, \quad u_{11i}=0 .$$
 Choosing the remaining coordinates $e_j$, $j\ge2$ so that
$D^2u$ is diagonal at $x$, we have 
\[
w^2=\sum_{i=1}^n u_{ii}^2, \quad w_k=\sum_{i=1}^n \frac{u_{ii}}{w} \, u_{iik}.
\]
Thus by the Cauchy-Schwarz inequality, for each $k$,
$$ w_k^2 \le \sum_i \left (\frac{u_{ii}}{w}\right)^2 \sum_i u_{iik}^2=\sum_i  u_{iik}^2 .$$
 Then
 \begin{equation}\label{32}
 \sum_{i,j,k} u_{ijk}^2
 =\sum_{i,k} u_{iik}^2 +\sum_{i \ne j, k} u_{ijk}^2 \ge |\nabla w|^2 + 2 \sum_{i \ne k}u_{iik}^2.
 \end{equation}
 
 Next we estimate for each $k$ the sum in the last term above.
  If $k=1$, then
 \begin{equation}\label{33}
  \sum_{i \ne 1} u_{ii1}^2=\sum_{i \ne 1} \left(\frac{u_{ii}}{|x|}\right)^2=\frac{w^2}{|x|^2}.
 \end{equation}
 If $k \ne 1$, then we use  $\triangle u_k=0$ and $u_{11k}=0$ and obtain
 $$\sum_{i \ne k}u_{iik}=-u_{kkk} \quad \Longrightarrow  \quad (n-2)\sum_{i \ne k} u_{iik}^2 \ge u_{kkk}^2, \quad \mbox{$k\neq1$, fixed.}$$
Now, substituting
 $$u_{iik}^2=\frac{1}{n-1}u^2_{iik}+\frac{n-2}{n-1}u_{iik}^2, $$
we find that 
 \begin{equation}\label{34}
 \sum_{i \ne k} u_{iik}^2 \ge \frac{1}{n-1}\sum_{i \ne k}u^2_{iik} + \frac{1}{n-1}u^2_{kkk} = \frac{1}{n-1} \sum_i u^2_{iik} \ge \frac{1}{n-1} w_k^2, \quad \mbox{$k\neq1$, fixed.}
 \end{equation}

Using \eqref{32}-\eqref{34} in \eqref{31}, we have
\[
w\triangle w = \sum_{i,j,k} u_{ijk}^2 - |\nabla w|^2 
\ge 2 \sum_{i\neq k} u_{iik}^2 \ge 2\frac{w^2}{|x|^2} + \frac1{n-1} \sum_{k=2}^n w_k^2
\]
and remarking that $$w_1=-\frac{w}{|x|},$$ since $w$ is homogeneous of degree $-1$,
the inequality of Proposition \ref{SI} is established. 
 \end{proof}
 
 \begin{cor} \label{CSI} The function $\bar v=w^\alpha$, which is homogeneous of degree $-\alpha$, satisfies
 \begin{equation}\label{35}
 \triangle \bar v \ge \alpha(\alpha +1) \frac{\bar v}{|x|^2}  \quad \quad \mbox{for all} \quad  \alpha \ge 1-\frac{2}{n-1},
 \end{equation}
 \end{cor}
 \begin{proof}    The conclusion of Proposition \ref{SI}
 can be written as
 $$\triangle (\log w) \ge \left(\frac{2}{n-1}-1  \right)|\nabla (\log w)|^2+ 2 \frac{n-2}{n-1} \frac{1}{|x|^2},$$
 or
 $$\triangle (\alpha \log w) + |\nabla (\alpha \log w)|^2 \ge \alpha \left(\frac{2}{n-1}-1 +\alpha \right)|\nabla (\log w)|^2+ 2 \alpha \frac{n-2}{n-1} \frac{1}{|x|^2}.$$
Using  $$|\nabla (\log w)| \ge \frac{w_1^2}{w^2}=\frac{1}{|x|^2},$$
and $\alpha \ge 1 - 2/(n-1)$, we see that 
 \begin{align*}
  \triangle (\alpha \log w) + |\nabla (\alpha \log w)|^2 & \ge \alpha \left(\frac{2}{n-1}-1 +\alpha + 2 \frac{n-2}{n-1}\right) \frac{1}{|x|^2} \\
  & = \alpha(\alpha+1) \frac{1}{|x|^2}.
  \end{align*}
It follows that \eqref{35} holds on the set $\bar v>0$.  On the other
hand, at points of $\{\bar v = 0\}\cap \Om$, $\Delta \bar v\ge0$ holds
in the sense of viscosity since  $\bar v\ge0$.
\end{proof}

  \subsection{ The boundary inequality} 
  We have $$w^2=\sum_{i,j} u_{ij}^2  \quad \Longrightarrow \quad w \, w_n=
  \sum_{i,j} u_{ij} \, u_{ijn} .$$
  
  Fix a point $x_0$ on $\p \Om \setminus \{0\}$ and choose a system of coordinates as in Section 2, i.e. such that $e_n=\nu_{x_0}$, $D^2 u(x_0)$ is diagonal, and $e_1$ coincides with the radial direction $x_0/|x_0|$. We recall \eqref{2}
   \begin{align*}
  u_{iin}&= u_{nn} u_{ii} -u_{ii}^2  \quad \quad \quad \quad \mbox{for all $i<n$,} \\
  u_{nnn}&= w^2. %= -Hu_{nn}+w^2-u_{nn}^2,
  \end{align*}
Thus, using $u_{nn}=-H$,
\begin{align*}
w\, w_n  & = u_{nn}w^2 + \sum_{i<n} (u_{nn}u_{ii}^2 - u_{ii}^3) \\
& = u_{nn} w^2 + \sum_{k=1}^n (u_{nn}u_{kk}^2 - u_{kk}^3) \\
& =  -2 H w^2 -\sum_{k=1}^n u^3_{kk}.
\end{align*}
Therefore,
\[
\frac 1 H (\log w)_n = -\left(2+ \frac{1} {H w^2}  \sum_{k=1}^n u_{kk}^3 \right) 
\ge - L
\]
where 
   \begin{equation}\label{L}
   L:=\max _{\p \Om_S} 
   \left(2+ \frac{1} {H w^2}  \sum_{k=1}^n u_{kk}^3 \right) 
   \end{equation}
   and we see that the function $\bar v= w^\alpha$ satisfies
   \begin{equation}\label{36}
   \frac 1 H (\log \bar v)_\nu \ge -1 \quad \quad \mbox{if} \quad \alpha \le \frac 1 L.
   \end{equation}
   
   From \eqref{35}, \eqref{36} and Proposition \ref{p1} we see that $u$ is unstable  if there exists $\alpha$ such that
   \[
      \alpha \ge 1 -\frac {2}{n-1}, \quad \quad \alpha \le \frac 1 L,\quad
      \mbox{and} \qquad 
\alpha(\alpha +1) > \left( \frac n 2 -1-\alpha \right)^2.
\]
Notice that the second lower bound on $\alpha$ guarantees the first lower bound 
since the second lower bound is equivalent to 
\[
\alpha > \frac {(n- 2)^2}{4(n-1)} \ge 1-\frac{2}{n-1}.
\]
  
   We summarize these results in the next proposition.
     \begin{prop}\label{p2}
   Let $u$ be a solution to \eqref{op} which is not one-dimensional. Then $u$ is unstable if 
   \begin{equation}\label{37}
   \frac {(n- 2)^2}{4(n-1)}  < \frac 1 L,
   \end{equation}
   with $L$ given by \eqref{L}.  Moreover, $u$ is unstable also in case of equality in \eqref{37} provided that equality does not hold at all points in  \eqref{35}, \eqref{36} .
   \end{prop} 
   \begin{cor}
   If $u$ is a stable solution  to \eqref{op} in dimension $n = 3$ then $u$ is one-dimensional. 
   \end{cor}
   \begin{proof}  When $n=3$ the left side of \eqref{37} is $1/8$, while
   $L=2$ since in our coordinate system $u_{11}=0$ and $u_{22} = -u_{33}$. 
      \end{proof}   

Unfortunately \eqref{37} need not be true in dimensions $4 \le n\le 6$.   
To see this we express $L$ at $x_0 \in \p \Om$ in terms
of the relative sizes of the $n-2$ nonvanishing curvatures of $\p \Om$ 
at that point. Let $\kappa_\ell$, 
   $\ell=2,..,n-1,$ denote the curvatures of $\p \Omega$ with respect to the outer normal, ($\kappa_1=0$ since $e_1$ is the radial direction).   Define 
   \[
   \mu_\ell:=\frac{\kappa_\ell}{H} \implies u_{\ell \ell}  = \mu_\ell H, \quad u_{nn} = -H
   \]
Recall that $H = \sum_{\ell < n} \kappa_\ell >0$.  Since $\triangle u = 0$,
\[
\sum_{\ell  =2}^{n-1}  \mu_\ell=1
\]
Thus an upper bound for $L$ in \eqref{L} is given by 
   \begin{equation}\label{38}
L\le L^*  : = \sup \{2+ \frac{\sum  \mu_\ell^3-1}{1+ \sum \mu_\ell^2}:
\sum \mu_\ell=1\}   \qquad \mbox{(sums from $\ell = 2,\dots, n-1$).}
   \end{equation}
It is not hard to show that when $n=4$, $L^* = 7/2$,
whereas the left hand side in \eqref{37} is $1/3$.  
Moreover, if $n \ge 5$, then $L^*  = \infty$.
   
We remark however that condition \eqref{37} gives the sharp result in
the case when all curvatures are equal, i.e. the axis symmetric
case. Then $L = L^* =(n-1)/(n-2)$ and \eqref{37} holds for $n\le 5$. When
$n=6$ we have equality in \eqref{37}, but in this case, the
equality in \eqref{35} is strict.  Indeed, choosing
$$\alpha=\frac 1L = \frac45 > \frac35 = 1-\frac{2}{n-1},$$ 
the computation in the proof of Corollary \ref{CSI}  shows that
\[
\triangle (\alpha \log w) + |\nabla (\alpha \log w)|^2 \ge \alpha (\alpha+1) \frac{1}{|x|^2} +  \alpha \left(\frac{2}{n-1}-1 +\alpha\right) \frac{w^2_n}{w^2}.
\]
The last term is positive on $\p \Om$ because $w_n/w = -HL <0$. 

Finally we point out the main difference with the minimal surface
theory.  Proposition \ref{p2} requires an exponent $\alpha$ 
satisfying
\[
(n-2)^2/4(n-1) \le \alpha \le 1/L
\]
The lower bound is essentially maximized when $D^2u$ has only one 
negative eigenvalue and 
the remaining ones are positive and equal (as in the axis symmetric
case). On the other hand, the upper bound is minimized 
(that is, $L \to \infty$, $n \ge 5$) when, on the 
boundary $\p \Om$, one tangential eigenvalue is positive and the
remaining ones are negative. In other words, the constraints on 
$\alpha$ that come from the interior inequality and boundary
inequalities are individually nearly optimal but they are 
achieved for different configurations. This is
one way to understand why \eqref{37} is not sufficient to
prove instability
in the conjectured optimal range, that is, for $n \le 6$.

 Our computation is somewhat consistent with the findings of G. Hong
 in \cite{H} where he studied the stability of Lawson-type cones for
 \eqref{op} in low dimensions. It turns out that in dimension $n=7$
 there are in fact two different stable cones corresponding precisely
 to the two situations described above.
 
   \section{The case $w^2=\sum_{\lambda_k>0} \lambda_k^2 + a \sum_{\lambda_k<0} \lambda_k^2$.}\label{s4}
   
   In this section we proceed as in Section 3 for a different choice of $w$ i.e.
   \begin{equation}\label{w}
   w^2:=\sum_{\lambda_k>0} \lambda_k^2 + a \sum_{\lambda_k<0} \lambda_k^2,
   \end{equation}
   for some constant $a >0$. Here $\lambda_i$ represent the eigenvalues of $D^2u$. 
   
When $a=1$ then $w$ coincides with the function considered in
Section 3. We show that when $a=4$ and $n=4$, the interior
inequality remains the same as in Section 3, however the boundary
inequality improves from $L\le 7/2$ to $L \le 3$ and allows us to prove Theorem
\ref{d4}.

   \subsection{Functions of the eigenvalues}
   
Assume $$F(D^2u)=f(\lambda_1,..,\lambda_n),$$
   with $f \in C^1$ a symmetric function of its arguments. We choose a system of coordinates at a point $x \in \Om$ such that 
   $$D^2 u=diag(\lambda_1,..,\lambda_n),$$ and we use the following orthonormal basis in the space of symmetric matrices
   $$e_{ii}:=e_i \otimes e_i, \quad \quad e_{ij}:=\frac{1}{\sqrt 2} (e_i \otimes e_j + e_j \otimes e_i) \quad \mbox{for $i<j$}.$$
   Then one can check that
   \begin{equation}\label{41}
   F_{e_{ii}}(D^2 u)=f_{\lambda_i} \quad \mbox{and} \quad F_{e_{ij}}(D^2u)=0.
   \end{equation}
   Moreover, if $f \in C^2$ then
   $$F_{e_{ii},e_{kk}}(D^2u)=f_{\lambda_i \lambda_k} ,$$
  
   $$F_{e_{ij},e_{ij}}(D^2u)=\left\{
   \begin{array} {l} 
    \,  \frac{f_{\lambda_i}-f_{\lambda_j}}{\lambda_i-\lambda_j} \quad \quad \mbox{if $\lambda_i \ne \lambda_j$,} \\
   \   \\
   f_{\lambda_i \lambda_i} \quad \quad \quad \quad \mbox{if $\lambda_i=\lambda_j$.}
   \end{array}
   \right.   $$
  
   $$F_{e_{ij},e_{kl}}(D^2u)=0 \quad \mbox{if} \quad e_{ij} \ne e_{kl}, \, i<j.$$
   
   These can be checked from the fact that the eigenvalues of the matrix 
   $$\begin{pmatrix} \lambda_1 & \eps \\ \eps & \lambda_2 \end{pmatrix}  $$
   are $$\lambda_1+ \frac{\eps^2}{\lambda_1-\lambda_2} + O(\eps^3) \quad \mbox{and} \quad  \lambda_2+ \frac{\eps^2}{\lambda_2-\lambda_1} + O(\eps^3) \quad \mbox{if $\lambda_1 \ne \lambda_2$}$$
or $$\lambda_1 + \eps, \quad \lambda_2-\eps \quad \mbox{if $\lambda_1=\lambda_2$} .$$

\

\subsection{The interior inequality}
We show that the function $w$ defined in \eqref{w} satisfies the same differential inequality as in Proposition \ref{SI}. Rather surprisingly we can prove a more general statement: any convex, symmetric, homogeneous of degree one function of the eigenvalues satisfies the same conclusion as  Proposition \ref{SI}.  
   
 \begin{thm}\label{p3}
 Assume $\triangle u =0$ and let $$w=F(D^2u):=f(\lambda_1,..,\lambda_n),$$
 with $f$ a convex, symmetric, homogeneous of degree one function. Then
 $$w  \, \triangle w \ge \frac 2 n \, \,  |\nabla w|^2.$$
  Moreover, if $u$ is homogeneous of degree 1, the inequality can be improved to
 $$w \triangle w \ge \frac{2}{n-1}|\nabla w|^2 + 2 \, \frac{n-2}{n-1} \frac{w^2}{|x|^2} .$$
(The inequalities above are understood in the viscosity sense.)
 \end{thm}

   We remark that the hypotheses on $f$ easily imply $f \ge 0$. 
  Notice that the first inequality is equivalent to $w^{1-\frac 2n}$ is subharmonic, or in the case $n=2$ that $\log w$ is subharmonic.
  
  \begin{proof}
  
  We assume that $f$ is smooth in $\R^n \setminus \{0\}$. Then the general case easily follows by approximation. Also, it suffices to show the inequality in the set $\{w>0\}$ since it is obvious in $\{w=0\}$.
  
  Fix a point $x$ with $D^2 u (x) \ne 0$, and we choose a system of coordinates at $x$ such that $$D^2u=diag(\lambda_1,..,\lambda_n).$$ 
  First we show that
  \begin{equation}\label{fij}
  (f_{\lambda_i}-f_{\lambda_j})(\lambda_i-\lambda_j) \ge 0,
  \end{equation}
  and the inequality is strict if $f$ is strictly convex and $\lambda_i \ne \lambda_j$. 
  
  Indeed, let $Z_0:=(\lambda_1,..,\lambda_n)$ and let $Z_1$ denote the vector obtained from $Z_0$ after interchanging $\lambda_i$ with $\lambda_j$. Using the symmetry and convexity of $f$ we obtain
  $$0=f(Z_1)-f(Z_0) \ge \nabla f (Z_0) \cdot (Z_1-Z_0),$$
  and this gives our claim \eqref{fij}.  
  
  In view of Section 4.1, for each $k$ we have
 \begin{equation}\label{42}
w_k=f_{\lambda_i} u_{iik},
\end{equation}
and
 \begin{equation*}\label{43}
w_{kk}=\sum_{i} f_{\lambda_i} u_{iikk} + \sum_{i,j} f_{\lambda_i \lambda_j} u_{iik} u_{jjk} + 2 \,\sum_{i<j}  F_{e_{ij},e_{ij}} \,  u_{ijk}^2,
\end{equation*}
Summing over $k$ and using that $f$ is convex and $\triangle u_{ii}=0$ we find
\[
\triangle w \ge 2  \sum_{i<j} \sum_k  F_{e_{ij},e_{ij}} u_{ijk}^2
\]
Notice that all such terms are nonnegative since, by \eqref{fij}, 
$F_{e_{ij},e_{ij}} \ge 0$. We keep only the terms for which either $i=k$ or
$j=k$ and obtain
$$
\triangle w \ge 2  \sum_{i \ne j}F_{e_{ij},e_{ij}} u_{iij}^2,
$$
where $i$, $j$ run over $\{1,..,n\}$ with $i \ne j$. 

In order to obtain our inequality it suffices to show that 
\begin{equation}\label{44}
 \sum_{i \ne k}F_{e_{ik},e_{ik}} u_{iik}^2  \ge \frac 1n \frac{w_k^2}{w} \quad 
\mbox{for each fixed $k$.}
  \end{equation}
>From \eqref{42} and $\triangle u_k=0$ we find 
$$
w_k=\sum_{i \ne k} (f_{\lambda_i}-f_{\lambda_k}) u_{iik} \qquad \mbox{($k$ fixed)}.
$$
Notice that from Section 4.1 and the symmetry of $f$ we have 
$$
f_{\lambda_i}-f_{\lambda_k}= (\lambda_i-\lambda_k) F_{e_{ik},e_{ik}}.
$$ 
Hence by the Cauchy-Schwarz inequality and \eqref{fij} we obtain
$$
w_k^2 \le  \left(\sum_{i \ne k}F_{e_{ik},e_{ik}} u_{iik}^2 \right) \left( \sum_{i \ne k}  (\lambda_i-\lambda_k) (f_{\lambda_i}-f_{\lambda_k}) \right).
$$ 
Thus, in order to prove \eqref{44} it suffices to show that
\begin{equation}\label{45}
\sum_{i \ne k}  (\lambda_i-\lambda_k) (f_{\lambda_i}-f_{\lambda_k})   \le n f
\qquad \mbox{for each fixed $k$}
\end{equation}
  Indeed, using that $\sum \lambda_i=0$, $\sum \lambda_i f_{\lambda_i}=f$, 
and summing over both $i$ and $j$, we obtain the identity
\begin{equation}\label{46}  
  \sum_{i<j} (\lambda_i-\lambda_j) (f_{\lambda_i}-f_{\lambda_j})=\sum_i f_{\lambda_i} \sum_{j \ne i}(\lambda_i-\lambda_j)= \sum_i f_{\lambda_i} n\lambda_i =nf .
  \end{equation}
  %  \begin{align*}
%   \sum_{i \ne k}  (\lambda_i-\lambda_k) (f_{\lambda_i}-f_{\lambda_k}) & = n \lambda_k f_{\lambda_k} + \sum_{i \ne k} \left (\lambda_i + \sum_{j \ne k}\lambda_j \right ) f_{\lambda_i} \\
  % &= n \lambda_i f_{\lambda_i} + \sum_{i,j\ne k} (\lambda_j-\lambda_i)f_{\lambda_i}\\
  % &= n f - \frac 1 2 \sum_{i,j\ne k} (\lambda_j-\lambda_i) ( f_{\lambda_j}-f_{\lambda_i})
  % \end{align*}
Our claim \eqref{45} follows since, by \eqref{fij}, 
the left hand side in \eqref{45} is bounded above by the 
left hand side of \eqref{46}.
  
  \
  
{\it Remark:} From the equality above and \eqref{fij} we see that,
if $f$ is strictly convex in a neighborhood of
$Z_0=(\lambda_1,..,\lambda_n)$, we have equality in \eqref{45} only
when all $\lambda_i$ with $i \ne k$ are equal. In other words, the
coefficient of $1/n$ on ${w_k^2}/{w}$ in \eqref{44} can be replaced
by $1/n + \eps$ in a neighborhood of $x$, if $\lambda_i \ne
\lambda_j$ for some $i,j \ne k$ and $f$ is strictly convex near
$Z_0$ in the 2-dimensional plane generated by the $\lambda_i$,
$\lambda_j$ directions. Here $\eps>0$ depends on
$(\lambda_i-\lambda_j) (f_{\lambda_i}-f_{\lambda_j})$.
  
  \
  
  We conclude with the case when $u$ is homogeneous of degree 1 and show that the inequalities above can be improved. We assume that at the point $x$, the $e_1$ direction represents the radial direction $x/|x|$, thus $$\lambda_1=0, \quad u_{ij1}=-\frac{u_{ij}}{|x|}.$$ 

  Let $k \ne 1$. Then, the coefficient of $1/n$ in \eqref{44} 
can be replaced by $1/(n-1)$. Indeed, $u_{11k}=0$, $\lambda_1=0$, 
thus  the index $i=1$ can be ignored in the computations above, 
and we reduce the problem to $n-1$ variables.

When $k=1$ the left hand side of \eqref{45} equals $f$ since
  $$ \sum_{i \ne 1} (\lambda_i-\lambda_1)
(f_{\lambda_i}-f_{\lambda_1})= \sum_{i \ne 1} \lambda_if_{\lambda_i} =
f,$$ where we used $\lambda_1=0$, $\sum \lambda_i=0$. This shows that
the coefficient of $w_1^2/w^2$ in \eqref{44} can be replaced by
1. Since $w$ is homogeneous of degree $-1$ we also have
$w_1=-w/|x|$,  and we obtain
\[
w \Delta w \ge \frac{2}{n-1} \sum_{k=2}^n w_k^2 \ + w_1^2 = 
\frac{2}{n-1}|\nabla w|^2 + 2\frac{n-2}{n-1} \frac{w^2}{|x|^2}.
\]
\end{proof}
  
  \subsection{The boundary inequality} We show that the function $w$ defined in \eqref{w}, when $a=4$, $n=4$ satisfies
  \begin{equation}\label{47}
  \frac{1}{H} (\log w)_\nu \ge -3.
  \end{equation}
  
  Notice that $w \in C^{1,1}$ in the set $\{ w \ne 0\}$.
  
Let $x_0 \in \p \Om \setminus\{0\}$ and we choose a system of
coordinates as before i.e. with $D^2u(x_0)$ diagonal, $e_n=\nu_{x_0}$
and $e_1=x_0/|x_0|$. Denote by $i$ and $s$ the indices for
which $\lambda_i>0$ and $\lambda_s<0$, respectively.
>From \eqref{42}, \eqref{2} and $\lambda_n = -H$, we have
  \begin{align*}
  w_n&  
=\sum_i \frac{\lambda_i}{w} u_{iin} + a  \sum_s \frac{\lambda_{s}}{w}u_{ssn} 
\\
&  = \sum_i \frac{\lambda_i}{w} (-H\lambda_i-\lambda_i^2) + 
a \sum_{s \ne  n} \frac{\lambda_s}{w}(-H\lambda_s-\lambda_s^2) 
+a \frac{\lambda_n}{w}\sum_{k=1}^n \lambda_k^2 \\
& = -\frac{H}{w}
\left(\sum_i \lambda_i^2 + a\sum_{s\ne n} \lambda_s^2 + a \lambda_n^2\right)
- \frac1w 
\left( \sum_i \lambda_i^3 + a\sum_{s\ne n} \lambda_s^3 + a \lambda_n^3\right)
- \frac{aH}{w} \sum_{k=1}^n \lambda_k^2.
  \end{align*}
%where we used $\lambda_n=-H$ in the last step.
Multiplying by $-1/Hw$, we write this in more compact form as
\[
-\frac{w_n}{Hw} =
1 + \frac{\sum \lambda_i^3+a \sum \lambda_s^3 + a H \sum \lambda_k^2}{Hw^2}.
\]
where the $k$ is summed over all indices and, as before,
$i$ is summed over indices for which $\lambda_i>0$, and $s$ is
summed over indices for which $\lambda_s <0$. 

Since $\lambda_1=0$ and $\lambda_4<0$, we distinguish two cases depending 
whether $\lambda_2$ and $\lambda_3$ are both positive or have opposite signs.

\

{\it Case 1:} $\lambda_2>0$, $\lambda_3 \le 0$. 

Let $$\mu:=-\frac{\lambda_3}{H} \quad \mbox{thus} \quad \frac{\lambda_2}{H}=\mu +1, \quad \mbox{and $\mu \ge 0$}. $$

We need to show that
$$\frac{(1+\mu)^3-a\mu^3+a((1+\mu)^2+\mu^2)}{(1+\mu)^2+a\mu^2+a} \le 2 .$$
This is equivalent to
$$(\mu-1) \left(a (\mu^2 + \mu -1) -(\mu+1)^2 \right) \ge 0,$$
or, since $a=4$,
$$(\mu-1)^2(3 \mu + 5 ) \ge 0, $$
which is obvious since $\mu \ge 0$.

\

{\it Case 2:}  $\lambda_2>0$, $\lambda_3>0$.

Let $$\mu:=\frac{\lambda_2}{H} \quad \mbox{thus} \quad \frac{\lambda_3}{H}=1-\mu, \quad \mbox{and} \quad \mu\in (0,1). $$
We need to show that
$$\frac{\mu^3+(1-\mu)^3+4(\mu^2+(1-\mu)^2)}{\mu^2+(1-\mu)^2+4} \le 2.$$
This is obvious since the numerator is bounded above by $5$ thus the fraction is bounded by $5/4 <2$.

In conclusion \eqref{47} is proved and equality at a point holds only when 
\begin{equation}\label{48}
 \lambda_2>0, \quad \quad \lambda_3=\lambda_4<0.
 \end{equation}

\

{\it Proof of Theorem \ref{d4}}. Let $n=4$ and set $$\bar v:=w^\frac 13$$ with $w$ as in \eqref{w} and $a=4$. Assume that $w$ is not identically $0$, i.e. $u$ is not  a one-dimensional solution.

By Theorem \ref{p3} and \eqref{47} it follows as in Section 3 that $\bar v$ satisfies \eqref{si2}. In order to prove that $u$ is not stable it remains to show that $\bar v$ is a strict subsolution. 

We fix a point $x_0 \in \p \Om$. If equality holds in \eqref{47} then, by \eqref{48}, $\lambda_2 \ne \lambda_3$ at $x_0$. 
Then, from the remark in the proof of Theorem \ref{p3}, it follows that the differential inequality can be improved by adding a term $\eps w_n^2$ to the right hand side.
 Since $w_n=-3Hw<0$, we find that at $x_0$ we have strict inequality in Theorem \ref{p3} which in turn gives that $\bar v$ is a strict subsolution for the interior problem in a neighborhood of $x_0$.
 
 \qed

\end{document}